\documentclass[12pt]{amsart}
\usepackage{amsmath,amssymb,amsfonts,amsthm,amsopn}

\usepackage{graphicx,tikz}
\usepackage[all]{xy}
\usepackage{amsmath}
\usepackage{multirow, longtable, makecell, caption, array}
\usepackage{amssymb}
\usepackage{mathtools}
\usepackage{pb-diagram}
\usepackage{cellspace} %
\newcommand{\modsp}{modulation space}
\newcommand{\tf}{time-frequency}

\newcommand{\aw}{Anti-Wick}
\newcommand{\aaf}{A_a^{\varphi_1,\varphi_2}}

\newtheorem{tm}{Theorem}[section]
\newtheorem{lemma}[tm]{Lemma}
\newtheorem{prop}[tm]{Proposition}

\newtheorem{theorem}{Theorem}[section]
\newtheorem{corollary}[theorem]{Corollary}
\newtheorem{definition}[theorem]{Definition}

\newtheorem{proposition}[theorem]{Proposition}
\newtheorem{remark}[theorem]{Remark}
\newtheorem*{thm*}{Theorem}
\newtheorem*{cor*}{Corollary}
\newcommand{\beqa}{\begin{eqnarray*}}
	\newcommand{\eeqa}{\end{eqnarray*}}

\newcommand{\field}[1]{\mathbb{#1}}
\newcommand{\bR}{\field{R}}
\newcommand{\bN}{\field{N}}
\newcommand{\bZ}{\field{Z}}
\newcommand{\bC}{\field{C}}

\def\G{\mathcal{G}}

\def\la{\lambda}

\def\cF{\mathcal{F}}
\def\cS{\mathcal{S}}

\def\a{\aleph}

\def\rd{\bR^d}

\def\rdd{{\bR^{2d}}}
\def\zdd{{\bZ^{2d}}}

\def\lrd{L^2(\rd)}

\def\zd{\bZ^d}

\def\intrd{\int_{\rd}}
\def\intrdd{\int_{\rdd}}

\def\R{\right)}

\def\<{\left<}
\def\>{\right>}

\def\inv{^{-1}}

\def\mv1{M_v^1}

\def\Mmpq{M_m^{p,q}}
\def\phas{(x,\omega )}

\def\o{\omega}
\def\a{\alpha}
\def\b{\beta}

\def\N{\mathbb{N}}
\def\R{\mathbb{R}}
\def\Ren{\mathbb{R}^d}
\def\Renn{\mathbb{R}^{2d}}

\def\Qs{{Q_s}}

\def\sch{\mathcal{S}}

\def\Fur{\mathcal{F}}

\def\f{\varphi}

\def\gaw{A_a^{\f_1,\f_2}}

\def\Sn2{S_{2}(L^{2}(\Ren))}
\def\S1{S_{1}(L^{2}(\Ren))}
\def\sig00{\sigma_{0,0}}

\def\la{\langle}
\def\ra{\rangle}

\def\vs{v_s}

\begin{document}
	\begin{abstract} We study decay and smoothness properties for eigenfunctions of compact localization operators $\gaw$. Operators $\gaw$ with symbols $a$ in the wide modulation space  $M^{p,\infty}$ (containing the Lebesgue space $L^p$), $p<\infty$, and  windows $\f_1,\f_2$ in the Schwartz class $\cS$ are known to be compact. We show that their $L^2$-eigenfuctions with non-zero eigenvalues are indeed highly compressed onto a few Gabor atoms. Similarly, for symbols $a$ in the weighted modulation spaces $M^{\infty}_{v_s\otimes 1} (\rdd)$, $s>0$ (subspaces of $M^{p,\infty}(\rdd)$, $p>2d/s$) the  $L^2$-eigenfunctions of $\gaw$ are actually Schwartz functions.\par
	An important role is played by quasi-Banach Wiener amalgam and modulation spaces.
	As a tool, new convolution relations for modulation spaces and multiplication relations for Wiener amalgam spaces in the quasi-Banach setting are exhibited.
	\end{abstract}
	
	\title{Decay and Smoothness for Eigenfunctions of Localization Operators}
	\author{Federico Bastianoni}
	\address{Dipartimento di Matematica, Universit\`a di Torino,  via Carlo Alberto 10, 10123 Torino, Italy}
	\email{federico.bastianoni@unito.it}
	\thanks{}
	\author{Elena Cordero}
	\address{Dipartimento di Matematica, Universit\`a di Torino, via Carlo Alberto 10, 10123 Torino, Italy}
	\email{elena.cordero@unito.it}
	\thanks{}
	\author{Fabio Nicola}
	\address{Dipartimento di Scienze Matematiche, Politecnico di Torino, corso
		Duca degli Abruzzi 24, 10129 Torino, Italy}
	\email{fabio.nicola@polito.it}
	\thanks{}
	
	\date{}
\maketitle

\section{Introduction }

The study of localization operators has a long-standing tradition. They have become popular with the papers by I. Daubechies \cite{DB1,DB2} and from then widely investigated by several  authors in different fields of mathematics: from signal analysis to pseudodifferential calculus, see, for instance \cite{Abreu2012,Abreu2016,BG2015,EleCharly2003,ECSchattenloc2005,medit,CharlyToft2011,CharlyToft2013,Nenad2015,Nenad2016,Nenad2018,WongLocalization}. In quantum mechanics they were already known as  Anty-Wick operators, cf. \cite{Shubin91} and the references therein. 

Localization operators can be introduced via the time-frequency representation known as  short-time Fourier transform (STFT). Let us though introduce the STFT. Recall  first  the  modulation $M_{\omega}$ and translation $T_{x}$ operators of a function $f$ on $\rd$:
 \[
 M_{\omega}f\left(t\right)= e^{2\pi it \omega}f\left(t\right),\qquad T_{x}f\left(t\right)= f\left(t-x\right),\quad \omega,x\in\rd.
 \]
 For $z=(z_1,z_2)\in\rdd$, we define the \tf \, shift $\pi(z)=M_{z_2} T_{z_1}$. 
Fix $\f\in\cS(\rd)\setminus\{0\}$.  We define the short-time Fourier transform of a tempered distribution $f\in\cS'(\rd)$ as
 \begin{equation}\label{STFTdef}
 V_{\f}f\phas=\langle f,\pi\phas \f\rangle=\Fur (f\overline{T_x \f})(\omega)=\int_{\Ren}
 f(y)\, {\overline {\f(y-x)}} \, e^{-2\pi iy \o }\,dy.
 \end{equation}
 
The localization operator $\aaf $ with symbol $a$ and
windows $\f _1, \f _2$ is formally defined   to be
\begin{equation}
\label{eqi4}
\aaf f(t)=\int_{\Renn}a \phas V_{\f _1}f \phas M_\omega T_x \f _2 (t)
\,
dx d\omega .
\end{equation}
If $\f _1(t)  = \f _2 (t) = e^{-\pi t^2}$, then $A_a = \aaf $ is the
classical \aw\ operator and the mapping $a \mapsto \aaf  $ is  a quantization rule in quantum mechanics \cite{Berezin71,deGossonsymplectic2011,Shubin91,WongLocalization}.

If one considers a symbol $a$ in the Lebesgue space $L^q(\rdd)$ ($1\leq q<\infty$) and window functions $\f_1,\f_2$ in the Feichtinger's algebra $M^1(\rd)$ (see below for its definition) then the localization operator $\aaf $ is  in the Schatten class $S_q$ (see \cite{EleCharly2003}). This implies, in particular, that $\aaf $ is a bounded and compact operator on $L^2(\rd)$.

Sharp results for localization operators on modulation spaces  (Banach case)  were obtained in \cite{Wignersharp2018}, thus concluding the open issues related to this problem.

The focus of this paper is the properties of \emph{eigenfunctions}  of \emph{compact} localization operators.

The study of eigenvalues and eigenfunctions of a restrict class of compact and self-adjoint localization operators, namely of the type $A_{\chi_\Omega}^{\f,\f}$, where $\Omega$ is a compact domain of the \tf\, plane and the window $\f$ is in $L^2(\rd)$,  was pursued in \cite{Abreu2012,Abreu2016,Abreu2017}. 
The focus of the previous papers, as well as new recent contributions \cite{Luef1,Luef2}, extending the previous results,  is the asymptotic behavior of the eigenvalues, depending on the domain $\Omega$.

 Our perspective is different: properties of eigenfunctions of a compact localization operator having a general symbol $a$, without any requirement on the geometry of the function $a$. In particular, the symbol $a$  does not need to have a compact support. Besides, the related localization operator $\aaf$  is not necessarily  a self-adjoint operator (but $\aaf$ is compact). It is easy to check that  the adjoint of a localization operator is given by $$(\aaf)^*=A_{\bar a}^{\f_2,\f_1};$$
hence the self-adjointness property forces the choice $\f_1=\f_2$ and the symbol $a$ real valued, as for the case $A_{\chi_\Omega}^{\f,\f}$ mentioned above. 
Our framework can allow the use of two different windows  $\f_1$ and $\f_2$ to analyze and synthesize the signal $f$, respectively. Moreover, the symbol $a$ can be a complex-valued function. 

For a linear bounded operator $T$ on $\lrd$ we denote by $\sigma(T)$ the \emph{spectrum} of $T$, that is the set $\{\lambda\in\bC: \,T-\lambda I\, \mbox{is\,not\,invertible}\}$; in particular, the set $\sigma_P(T)$ denotes the \emph{point spectrum} of $T$, that is 
 $$\sigma_P(T)=\{\lambda\in\bC\,: \exists\,\mbox{a\,non-zero}\,f\,\in\lrd \,\mbox{such\,that} \,Tf=\lambda f\}.$$

 Under our assumptions, the operator $\aaf$ is a compact mapping on $\lrd$. Hence, the spectral theory for compact operators yields $$\sigma(\aaf)\setminus\{0\}=\sigma_P(\aaf)\setminus\{0\}.$$For sake of completeness, we recall that for compact operators on $\lrd$ the number zero is always in the spectrum $\sigma(\aaf)$  and that the point spectrum $\sigma_P(\aaf)\setminus\{0\}$ (possibly empty) is at most countably infinite.

A function $f\in\lrd$ is called an \emph{eigenfunction} of the operator $\aaf$ if there exists $\lambda\in \bC$ such that
$$ \aaf f=\lambda f.$$
 We are interested in the properties of eigenfuctions  of $\aaf$ related to eigenvalues $\lambda\in \sigma_P(\aaf)\setminus\{0\}$, whenever $ \sigma_P(\aaf)\setminus\{0\}\not=\emptyset$.  

To chase our goal, the new idea is to use properties of quasi-Banach modulation and Wiener amalgam spaces. Some of these issues are investigated for the first time  in this paper.

To give a flavour of these results, we recall the definition of modulation spaces (un-weighted case: for the more general  cases see below). 

Fix a non-zero window function $g\in\cS(\rd)$ and  $p,q\in (0,\infty)$. The modulation space $M^{p,q}(\rd)$ consists of all tempered distributions $f\in\cS'(\rd)$ such that 

$$\| f\|_{M^{p,q}}:=\|V_g f\|_{L^{p,q}}=\left(\intrd\left(\intrd  |V_g f\phas |^p dx\right)^{\frac qp} d\o \right)^{\frac1q}<\infty
$$
(with natural modifications when $p=\infty$ or $q=\infty$).
For $p,q\geq 1$ the function $\|\cdot\|_{M^{p,q}}$ is a norm, and different window functions $g\in\cS(\rd)$ yield equivalent norms, thus the same space. Modulation spaces  $M^{p,q}(\rd)$, $p,q\geq 1$, are Banach spaces,  invented by H. Feichtinger in  \cite{feichtinger-modulation}, where many of their properties were already investigated.   

We remark that a localization operator $\gaw$ with   symbol $a\in M^{p,\infty}$, $1\leq p<\infty$,  and windows $\f_1,\f_2\in M^1$,    belongs to the Schatten class $S_p$ and, in particular, it is a \emph{compact} operator on $\lrd$ (cf. \cite[Theorem 1]{ECSchattenloc2005}).

 The (quasi-)Banach space $M^{p,q}$, $p,q>0$, were first introduced and studied by Y.V. Galperin and S. Samarah in \cite{Galperin2004}, see the next section for more details. Roughly speaking, the mapping $f$ is in $M^{p,q}(\rd)$ if it locally behaves like a function in $\cF L^q(\rd)$ and  ``decays'' as a function in $L^p(\rd)$ at infinity.

In this paper we extend convolution relations for Banach modulation spaces exhibited in \cite{EleCharly2003, Toftweight2004} to the quasi-Banach ones, see Proposition \ref{mconvmp} below. This result seems remarkable by itself.

Such convolution relations will be crucial for our main result, that can be simplified as follows (cf. Theorem \ref{main}).

\begin{theorem}
	Consider a symbol $a\in M^{p,\infty}(\rdd)$,  $0<p< \infty$,  non-zero windows $\f_1,\f_2\in\cS(\rd)$ and assume that $\sigma_P(\aaf)\setminus\{0\}\not=\emptyset$. If $\lambda\in \sigma_P(\aaf)\setminus\{0\}$,   any eigenfunction $f\in\lrd$ with eigenvalue  $\lambda$  satisfies
	$f\in \bigcap_{\gamma>0} M^\gamma(\rd)$.
\end{theorem}

Roughly speaking, this means that  $L^2$ eigenfunctions of such compact localization operators reveal to be extremely well-localized. To make this statement more precise, we  use Gabor frames.  Let $\Lambda=\alpha\zd\times \beta\zd$  be a lattice of the time-frequency plane.
The set  of time-frequency shifts $\G(g,\Lambda)=\{\pi(\lambda)g:\
\lambda\in\Lambda\}$, for a  non-zero $g\in L^2(\rd)$, is called a
Gabor system. The set $\G(g,\Lambda)$   is
a Gabor frame if there exist constants $A,B>0$ such that
\begin{equation}\label{gaborframe}
A\|f\|_2^2\leq\sum_{\lambda\in\Lambda}|\langle f,\pi(\lambda)g\rangle|^2\leq B\|f\|^2_2,\qquad \forall f\in L^2(\rd).
\end{equation}
If $A=B=1$ we call $\G(g,\Lambda)$ a \emph{Parseval} Gabor frame. In this case \eqref{gaborframe} reduces to \begin{equation}\label{Parseval}
\|f\|_2^2=\sum_{\lambda\in\Lambda}|\langle f,\pi(\lambda)g\rangle|^2,\quad \forall f\in L^2(\rd).
\end{equation}

From now on we restrict to Parseval Gabor frames. If the eigenfunction $f$ above  satisfies  $f\in \bigcap_{\gamma>0} M^\gamma(\rd)$, then  $f$ is highly compressed onto a few Gabor atoms $\pi(\lambda)g$. Indeed,  
for $N\in \bN_+$, its $N$-term approximation error $\sigma_N(f)$ presents super-polynomial decay. We define
\begin{equation}\label{Sigma}
\Sigma_N=\{ p=\sum_{k,n\in F} c_{k,n} \pi(\a k,\beta n)g\,:\,\, c_{k,n}\in\bC,\,\,F\subset \zd\times \zd, \,\,\mbox{card}\,F\leq N \}
\end{equation}
 (the set of all linear combinations of Gabor atoms consisting of at most $N$ terms). Note that $\Sigma_N$ is not a linear subspace since $\Sigma_N+\Sigma_N=\Sigma_{2N}$. That is why the approximation of a signal $f$ by elements of $\Sigma_N$ is often  referred to as non-linear approximation.
Given a function $f\in\lrd$, the $N$-term approximation error in $\lrd$ is 
\begin{equation}\label{sigma}
\sigma_N(f)=\inf_{p\in\Sigma_N} \|f-p\|_2.
\end{equation}
That is, $\sigma_N(f)$ is the error produced when $f$ is approximated optimally by a linear combination of $N$ Gabor atoms.  

We shall show in Corollary \ref{bo1} that, if $f\in \bigcap_{\gamma>0} M^\gamma(\rd)$, then,  for every $r>0$ there exists a positive constant $C=C(r,f)$ such that 

$$\sigma_N(f)\leq C N^{-r}.
$$

Since $\sigma_N(f)$ is the error produced when $f$ is approximated optimally by a linear combination of $N$ Gabor atoms, the decay above shows the high compression of the eigenfunction onto such atoms, cf. \cite[Subsection 12.4]{Grochenig_2001_Foundations}.\par
Another main result (Theorem \ref{main2}) states, in the same spirit, that for symbols in the weighted modulation space $M^{\infty}_{v_s\otimes 1} (\rdd)$, $s>0$ (see Definition \ref{def2.4} below), the corresponding $L^2$ eigenfunctions of $\aaf$, with eigenvalues $\lambda\not=0$, are actually in $\cS(\rd)$. Also in this case $\aaf$ is a compact operator, since $M^{\infty}_{v_s\otimes 1} (\rdd) \subset M^{p,\infty}(\rdd)$, whenever $p>2d/s$, cf. the subsequent Remark \ref{compact}. These symbol classes include certain measures and the result applies, in particular, to Gabor multipliers (see e.g., \cite{FGM2003}). We leave the precise statement to the interested reader.\par
In short, the paper is organized as follows.  Section 2 is devoted to the function spaces involved in our study. In particular, we show new multiplication relations for Wiener amalgam spaces in the quasi-Banach setting and we prove convolution relations for quasi-Banach modulation spaces. Section 3 represents the core of the paper. 
 We first exhibit continuity results for Weyl operators on modulation spaces (involving the quasi-Banach setting), cf. Theorem \ref{Charpseudo}. Then we study the eigenfunctions' properties for Weyl operators (Propositions \ref{eigenfunctWeyl} and \ref{eigenfunctWeyl2} below). 
 Next we show our main result on the eigenfunctions' regularity and smoothness of localization operators: Theorem \ref{main} and related consequences in terms of the $N$-term approximation (Proposition \ref{approx} and Corollary \ref{bo1}).
 The last Section $4$ is devoted to the study of eigenfunctions of localization operators with symbols in the weighted  Lebesgue spaces $L^q_m(\rd)$, $1\leq q<\infty$. 
 This  section suggests a wider study of the topic for localization operators on groups \cite{WongLocalization}. We shall pursue this issue in a subsequent paper.
 Let us conclude this introduction with another open problem, that can be worth investigating.
 Namely, in this paper we work with compact localization operators having symbols $a$ in $M^{\infty}_{v_s\otimes1}(\rdd)$, for $s>0$, or in the larger space  $M^{p,\infty}(\rdd)$, $0<p<\infty$. Notice that the symbol classes $M^{p,\infty}(\rdd)$ do not exhaust the class of compact localization operators $\aaf$, characterized in \cite[Theorem 3.15]{FG2006} (see also the contributions \cite{FG2007,FGP2017}); in fact, the localization operators object of study  belong to the Schatten class $S_p$, $p<\infty$ (see the characterization in \cite[Theorem 1]{ECSchattenloc2005}), that is a proper linear subspace of the space of compact operators. More comments on this topic are contained in the subsequent Remark \ref{schatten-comp}.

\section{Preliminaries}
 \textbf{Notation.} We define   $xy=x\cdot y$,  the scalar product on $\Ren$. The Schwartz class is denoted by  $\sch(\Ren)$, the space of tempered
  distributions by  $\sch'(\Ren)$.   We use the brackets  $\la
  f,g\ra$ to denote the extension to $\sch' (\Ren)\times\sch (\Ren)$ of
  the inner product $\la f,g\ra=\int f(t){\overline {g(t)}}dt$ on
  $L^2(\Ren)$. 
  The Fourier transform of a function $f$ on $\rd$ is normalized as
  \[
  \Fur f(\o)= \int_{\rd} e^{-2\pi i x\o} f(x)\, dx.
  \]
 The involution $g^*$ is given by $g^*(t) =
 \overline{g(-t) }$. Given a measurable and positive weight function $m$  on $\rdd$ and $1\leq p< \infty$, we denote by $L^p_m(\rd)$  the Banach space of measurable functions $f: \rd \to \bC$ satisfying 
 $$\|f\|_{L^p_m}=\left(\intrd |f(x)|^p m(x)^p\, dx\right)^{1/p},$$
 up to the equivalence relation $f\sim g$ if and only if $f(x)=g(x)\,$ for  a.e. $x$.
For $p=\infty$, we have $f\in L^\infty_m(\rd)$ if $\|f\|_{L^\infty_m }=$ ess sup$_{x\in\rd} |f(x)| m(x)<\infty$, again up to the equivalence relation above. For $s\in\bR$, recall the Sobolev spaces $H^s(\Ren)=\{f\,:\,\hat{f}(\o)\la \o\ra^s\in
L^2(\Ren)\},$ where $\la \o\ra^s =(1+|\o|^2)^{s/2}$.

We use $T^*$  for the adjoint of an operator $T$. Observe that  $\pi(z)^\ast=\pi(-z)$ and the following commutation relations hold 
\begin{equation}\label{CR}
\pi(z)\pi(w)=e^{-2\pi i z_1 w_2}\pi(w)\pi(z).
\end{equation}
For $f,g\in\lrd$, the cross-Wigner distribution is defined by:
\begin{equation}\label{CWD}
W(f,g)\phas=\intrd f(x+\frac t2)\overline{g(x-\frac t2)}e^{-2\pi i t\o}\,dt.
\end{equation} 
Given two normed spaces $A$ and $B$, we denote by $A\hookrightarrow B$ the continuous embedding of $A$ into $B$.

Finally we recall the definition of the Schatten classes $S_p$, $0<p<\infty$. The singular values
$\{s_k(L)\}_{k=1}^\infty$ of a compact operator $L$ on a Hilbert space are the eigenvalues of
the positive self-adjoint operator $(L^*L)^{1/2}$. The Schatten class $S_p$ consists
of all compact operators whose singular values lie in $\ell^p$.

\subsection{Weight functions}
In the sequel $v$ will always be a
continuous, positive,  submultiplicative  weight function on $\rd$ (or on $\zd$), i.e., 
$ v(z_1+z_2)\leq v(z_1)v(z_2)$, for all $ z_1,z_2\in\Ren$ (or for all $z_1,z_2\in \zd$).
We say that $m\in \mathcal{M}_v(\rd)$ (or $m\in \mathcal{M}_v(\zd)$) if $m$ is a positive, continuous  weight function  on $\Ren$ (or on $\zd$) {\it
	$v$-moderate}:
$ m(z_1+z_2)\leq Cv(z_1)m(z_2)$  for all $z_1,z_2\in\Ren$ (or for all $z_1,z_2\in \zd$).
We will mainly work with polynomial weights of the type
\begin{equation}\label{vs}
v_s(z)=\la z\ra^s =(1+|z|^2)^{s/2},\quad s\in\bR,\quad z\in\rd\,\, (\mbox{or}\, \zd).
\end{equation}
Observe that,  for $s<0$, $v_s$ is $v_{|s|}$-moderate.\par 
Given two weight functions $m_1,m_2$ on $\rd$ (or $\zd$), we write $$(m_1\otimes m_2)(x,\o)=m_1(x)m_2(\o),\quad x,\o\in \rd \quad (\mbox{or}\, \zd).$$

\subsection{Spaces of sequences}
\begin{definition}\label{ellp}
	For $0<p,q\leq \infty$, $m\in \mathcal{M}_v(\zdd)$, the space $\ell^{p,q}_m(\zdd)$ consists of all sequences $a=(a_{k,n})_{k,n\in\zd}$ for which the (quasi-)norm 
	$$\|a\|_{\ell^{p,q}_m}=\left(\sum_{n\in\zd}\left(\sum_{k\in\zd}|a_{k,n}|^p m(k,n)^p\right)^{\frac qp}\right)^{\frac 1q}
	$$
	(with obvious modification for $p=\infty$ or $q=\infty$) is finite.
\end{definition}

For $p=q$, $\ell^{p,q}_m(\zdd)=\ell^p_m(\zdd)$, the standard spaces of sequences.  Namely, in dimension $d$, 
 for $0<p\leq \infty$, $m$ a weight function on $\zd$, a sequence $a=(a_k)_{k\in\zd}$ is in $\ell^p_m(\zd)$ if
 $$\|a\|_{\ell^{p}_m}=\left(\sum_{k\in\zd}|a_{k}|^p m(k)^p\right)^{\frac 1p}<\infty.
 $$
 
 Here there are some properties we need in the sequel \cite{Galperin2014,Galperin2004}:
\begin{itemize}
	\item[(i)] \emph{Inclusion relations}: If $0<p_1\leq p_2 \leq\infty$, then $\ell^{p_1}_m(\zd)\hookrightarrow\ell^{p_2}_m(\zd)$, for any positive weight function $m$ on $\zd$.
	\item[(ii)] \emph{Young's convolution inequality}: Consider $m\in\mathcal{M}_v(\zd)$, $0<p,q,r\leq \infty$ with
	\begin{equation}\label{Yrgrande1}
	\frac1p+\frac1q=1+\frac1r, \quad \mbox{for}\quad 1\leq r\leq \infty
	\end{equation}
	and
	\begin{equation}\label{Yrminor1}
	p=q=r, \quad \mbox{for}\quad 0<r<1.
	\end{equation}
 Then for all $a\in \ell^p_m(\zd)$ and $b\in\ell^q_v(\zd)$, we have $a\ast b\in \ell^r_m(\zd)$, with
	\begin{equation}
	\|a\ast b\|_{\ell^r_m}\leq C \|a\|_{\ell^p_m}\|b\|_{\ell^q_v},
	\end{equation}
	where $C$ is independent of $p,q,r$, $a$ and $b$. If $m\equiv v\equiv1$, then $C=1$. 
	\item[(iii)] \emph{H\"{o}lder's inequality}: For any positive weight function $m$ on $\zd$, $0<p,q,r\leq \infty$, with $1/p+1/q=1/r$, 
	\begin{equation}\label{ptwellp}
	\ell^p_m (\zd)\cdot \ell^q_{1/m}(\zd)\hookrightarrow \ell^r(\zd),
	\end{equation}
	where the symbol $\hookrightarrow$ denotes that the inclusion is a  continuous mapping.
\end{itemize}
\subsection{Wiener Amalgam Spaces \cite{Feichtinger_1981_Banach,feichtinger-wiener-type,Feichtinger_1990_Generalized,Fournier_1985_Amalgams,Galperin2004,Heil-amalgam,Rauhut2007Coorbit,Rauhut2007Winer}.}

\begin{definition}\label{C2defWLpq}
	Consider $p,q\in (0,\infty]$, a weight function $m\in\mathcal{M}_{v}$
	and the compact set  $Q=[0,1]^d$. The Wiener amalgam
	space $W\left(L^{p},L_{m}^{q}\right)(\rd)$ consists of the functions $f\,:\,\mathbb{R}^{d}\rightarrow\mathbb{C}$
	such that $f\in L_{{loc}}^{p}(\rd)$ and for which the \emph{control
		function}:
	\begin{equation}\label{C2controlF}
	F^{Q}_f\left(k\right):=\left\Vert f\cdot T_{k}\chi_{Q}\right\Vert _{L^{p}}\in \ell_{m}^{q}\left(\mathbb{Z}^{d}\right),\qquad k\in\mathbb{Z}^{d}.
	\end{equation}
	The quasi-norm on $W\left(L^{p},L_{m}^{q}\right)$ is given by
	\begin{align}
	\left\Vert f\right\Vert _{W\left(L^{p},L_{m}^{q}\right)}&:=\left\Vert F_{f}^{Q}\left(k\right)\right\Vert _{\ell_{m}^{q}}\notag\\
	&=\left\Vert \left\Vert f\cdot T_{k}\chi_{Q}\right\Vert _{L^{p}}\right\Vert _{\ell_{m}^{q}}\notag\\
	&=\left(\sum_{k\in \mathbb{Z}^{d}}\left(\int_{\mathbb{R}^{d}}\left|f\left(t\right)\right|^{p}\chi_{Q}\left(t-k\right)\mathrm{d}t\right)^{\frac{q}{p}}m^{q}\left(k\right)\right)^{\frac{1}{q}}\label{C2normWLpLq},
	\end{align}
	with suitable adjustments for the cases $p,q=\infty$. 
\end{definition}

	This special definition allows us to grasp the sense of the amalgam:
		we first view $f$ \textquotedblleft locally\textquotedblright\, through translations $T_{k}\chi_{Q}$
		of the sharp cutoff function $\chi_{Q}$, and measure those
		local pieces in the $L^{p}$-norm, then we measure the global behavior
		of those local pieces according to the $\ell_{m}^{q}$-norm. The \textquotedblleft window\textquotedblright\,
		through which we view $f$ locally need not be a unit $d$-dimensional cube, cf. \cite{feichtinger-wiener-type,Galperin2004,Heil-amalgam,Rauhut2007Winer}. 
	In the sequel we shall use 	the following properties:
	\begin{itemize}
		\item [(i)] Inclusion relations: For $0<p_1\leq p_2\leq \infty$,  $0<q_2\leq q_1\leq \infty$, we have
		\begin{equation}\label{inclusionW}
		W(L^{p_2}, L^{q_2}_m)(\rd)\hookrightarrow W(L^{p_1}, L^{q_1}_m)(\rd).
		\end{equation}
		\item [(ii)] Convolution relations (for the quasi-Banach case see \cite[Lemma 2.9]{Galperin2004}): Consider $m_i\in\mathcal{M}_v$, $0<p_i,q_i\leq \infty$, $i\in \{1,2,3\}$, and $p_3\geq 1$.  Assume that $L^{p_1} \ast L^{p_2} \hookrightarrow L^{p_3}$ and $\ell^{q_1}_{m_1}\ast \ell^{q_2}_{m_2} \hookrightarrow \ell^{q_3}_{m_3}$, then
		\begin{equation}\label{convWiener}
		W(L^{p_1}, L^{q_1}_{m_1})\ast W(L^{p_2}, L^{q_2}_{m_2})\hookrightarrow W(L^{p_3}, L^{q_3}_{m_1}).
		\end{equation}
		\item[(iii)] For $m\in\mathcal{M}_v$, $0<p\leq\infty$, we have \begin{equation}\label{pWiener}
		L^p_m=W(L^p,L^p_m).
		\end{equation}
	\end{itemize}
	\begin{proposition}[Multiplication relations]\label{MR}
	 Consider $m,w\in\mathcal{M}_v$, $0<p_i,q_i\leq \infty$, $i=\{1,2,3\}.$ Assume $\frac 1{p_1}+\frac 1{p_2}=\frac 1{p_3}$ and $\frac 1{q_1}+ \frac 1{q_2} =\frac 1{q_3}$, then
		\begin{equation}\label{prodWiener}
		W(L^{p_1}, L^{q_1}_{m})\cdot W(L^{p_2}, L^{q_2}_{w/m})\hookrightarrow W(L^{p_3}, L^{q_3}_w).
		\end{equation}
	\end{proposition}
	\begin{proof} The result is well known for $1\leq p_i,q_i\leq \infty$, cf. \cite{Feichtinger_1981_Banach, Heil-amalgam}. Here we show that the same proof works for quasi-Banach spaces. Indeed, since the standard H\"older inequality holds for Lebesgue exponents in $(0,+\infty]$, for $f_1\in 	W(L^{p_1}, L^{q_1}_{m})$, $f_2\in 	W(L^{p_2}, L^{q_2}_{w/m})$ we have
$$
	\|f_1f_2 T_k\chi_Q\|_{L^{p_3}}=\|(f_1T_k\chi_Q)(f_2 T_k\chi_Q)\|_{L^{p_3}}\leq \|f_1 T_k\chi_Q\|_{L^{p_1}}\|\|f_2 T_k\chi_Q\|_{L^{p_2}}.
$$
   Defining $a_k=\|f_1 T_k\chi_Q\|_{p_1}$ and $b_k=\|f_2 T_k\chi_Q\|_{p_2}$ and  using  H\"{o}lder's inequality for sequences $\ell^{q_1}\ell^{q_2}\hookrightarrow\ell^{q_3}$, for $1/q_1+1/q_2=1/q_3$ ($0<q_i\leq \infty$, $i=1,2,3$), we obtain $$\|a_k b_k w(k)\|_{\ell^{q_3}}=\|(a_k m(k))( b_k w(k)/m(k))\|_{\ell^{q_3}}\leq \|a_k m(k)\|_{\ell^{q_1}}\|b_k w(k)/m(k)\|_{\ell^{q_2}}.$$
   This completes the proof.
	\end{proof}

	\subsection{Modulation Spaces} We use the extension to quasi-Banach spaces introduced first by Y.V. Galperin and S. Samarah in \cite{Galperin2004}.
\begin{definition}\label{def2.4}
	Fix a non-zero window $g\in\cS(\rd)$, a weight $m\in\mathcal{M}_v$ and $0<p,q\leq \infty$. The modulation space $M^{p,q}_m(\rd)$ consists of all tempered distributions $f\in\cS'(\rd)$ such that the (quasi-)norm 
	\begin{equation}\label{norm-mod}
	\|f\|_{M^{p,q}_m}=\|V_gf\|_{L^{p,q}_m}=\left(\intrd\left(\intrd |V_g f \phas|^p m\phas^p dx  \right)^{\frac qp}d\o\right)^\frac1q 
	\end{equation}
	(obvious changes with $p=\infty$ or $q=\infty)$ is finite. 
\end{definition}

The most famous modulation spaces  are those  $M^{p,q}_m(\rd)$ with $1\leq p,q\leq \infty$, invented by H. Feichtinger in \cite{feichtinger-modulation}. In that paper he proved  they are Banach spaces, whose norm does not depend on the window $g$, in the sense that different window functions in $\cS(\rd)$ yield equivalent norms. Moreover, the window class $\cS(\rd)$ can be extended  to the modulation space  $M^{1,1}_v(\rd)$ (so-called Feichtinger algebra). 

For shortness, we write $M^p_m(\rd)$ in place of $M^{p,p}_m(\rd)$ and $M^{p,q}(\rd)$ if $m\equiv 1$.

The modulation spaces $M^{p,q}_m(\rd)$, $0<p,q<1$, where  introduced almost twenty years later by Y.V. Galperin and S. Samarah in \cite{Galperin2004} and then studied in \cite{Kobayashi2006,Rauhut2007Coorbit,ToftquasiBanach2017,Wangbook2011} (see also references therein).
In this framework, it appears that the largest natural class of windows universally admissible for all spaces $M^{p,q}_m(\rd)$, $0<p,q\leq \infty$ (with $m$ having at most polynomial growth) is the Schwartz class $\cS(\rd)$.
There are thousands of papers involving modulation spaces with indices $1\leq p,q\leq\infty$, whereas very few works deal with the  quasi-Banach case $0<p,q\leq 1$. Indeed, many properties related to the latter case are still unexplored. 

In this paper their contribution is fundamental, since they are the key tool for  understanding the properties of eigenfunctions of localization operators having symbols with some decay at infinity, measured in the $L^p$-mean, $0< p<\infty$. \par 
We observe that for any $0<p\leq \infty$, the window function $g$ can be chosen in a bigger space than the Schwartz class $\cS(\rd)$, see below \cite[Proposition 1.2 (1)]{ToftquasiBanach2017}.
\begin{prop}\label{finestre}
	Consider $m\in\mathcal{M}_v$,  $r\in (0,1]$, $0<p\leq \infty$ such that  $r\leq p$.	Then the following is true: if $g\in M_v^r(\rd)\setminus\{0\}$,  then $f \in M^p_m(\rd)$
 if and only if \eqref{norm-mod} is finite. In particular, $M^p_m(\rd)$ 
	 is independent of the choice of $g\in M_v^r(\rd)\setminus\{0\}$.
	Moreover,  it is a quasi-Banach space under the quasi-norm in \eqref{norm-mod}, and different choices of $g$ give rise to equivalent quasi-norms.
\end{prop}
In the sequel we shall use inclusion relations for modulation spaces, we refer to \cite[Theorem 3.4]{Galperin2004} and \cite[Theorem 12.2.2]{Grochenig_2001_Foundations} for the proof of the following result.
\begin{theorem}\label{inclusionG}
	Let $m\in\mathcal{M}_v(\rdd)$. If $0<p_1\leq p_2\leq \infty$ and $0<q_1\leq q_2\leq \infty$ then $M^{p_1,q_1}_m(\rd)\hookrightarrow M^{p_2,q_2}_m(\rd)$.
\end{theorem}

	The duality properties for modulation spaces with indices $p,q<1$ where studied in \cite{Kobayashi2007} and completed in 
	\cite[Proposition 6.4, page 163]{Wangbook2011}:  
	\begin{proposition}\label{dual}
		Let $s\in\bR$ and $0<p,q<\infty$. If $p\geq 1$ we denote $1/p+1/p'=1$ (and similarly for $q$); if $0<p<1$ we denote $p'=\infty$. Then $(M^{p,q}_{1\otimes v_s}(\rd))^\prime=M^{p',q'}_{1\otimes v_{-s}}(\rd)$.
	\end{proposition}

\begin{remark}\label{compact}
In our framework it is important to notice the following inclusion relation for $s>0$:
$$ M^{\infty}_{v_s\otimes 1} (\rdd) \subset M^{p,\infty}(\rdd)\quad  \mbox{if}\, \,p>2d/s.$$
This follows from the recent contribution \cite[Theorem 1.5]{Guo2019}. Hence localization operators $\aaf$ with symbols $a$ in $ M^{\infty}_{v_s\otimes 1} (\rdd)$ and windows $\f_1,\f_2$ in $\cS(\rd)$ are compact operators belonging to the Schatten class $S_p$, with $p>2d/s$.
\end{remark}

We will repeatedly  use the following result, cf. \cite[Theorem 3.3]{Galperin2004} (see also \cite[Theorem 12.2.1]{Grochenig_2001_Foundations} for $p\geq 1$).
\begin{theorem}\label{G33}
Assume that $m\in\mathcal{M}_v$. For $0<p<1$ let  $g$ be a non-zero window in $M^r_v(\rd)$, $r\leq p$. For $1\leq p\leq \infty$, the function $g$ can be chosen in the larger space $M^1_v(\rd)$.  If $f\in M^p_m(\rd)$, $0<p\leq \infty$,  then $V_g f\in W(L^\infty, L^p_m) $ and there exists $C>0$, independent of $f$, such that
\begin{equation}
\|V_g f\|_{W(L^\infty, L^p_m)}\leq C \|V_g f\|_{L^p_m}.
\end{equation}
\end{theorem}

We also need to recall the inversion formula for
the STFT (see  \cite[Proposition 11.3.2]{Grochenig_2001_Foundations}): assume $g\in M^{1}_v(\rd)\setminus\{0\}$, $1\leq p,q\leq \infty$, 
$f\in M^{p,q}_m(\rd)$, then
\begin{equation}\label{invformula}
f=\frac1{\|g\|_2^2}\int_{\R^{2d}} V_g f(x,\o)M_\o T_x g\, dx\,d\o,
\end{equation}
and the  equality holds in $M^{p,q}_m(\rd)$.

\subsection{Gabor Frames}
Let $\Lambda=\alpha\zd\times \beta\zd$  be a lattice of the time-frequency plane.
 The set  of time-frequency shifts $\G(g,\Lambda)=\{\pi(\lambda)g:\
\lambda\in\Lambda\}$, for a  non-zero $g\in L^2(\rd)$, is called a
Gabor system. The set $\G(g,\Lambda)$   is
 a Gabor frame, if there exist constants $A,B>0$ such that
 \eqref{gaborframe} holds true.
 
 If $\G(g,\Lambda)$  is a Gabor frame, then it can be shown that the frame operator
 $$ Sf=\sum_{\lambda\in\Lambda}\langle f,\pi(\lambda)g\rangle\pi(\lambda)g,\quad f\in\lrd$$
 is a topological isomorphism on $\lrd$. Moreover, the system $\G(\gamma,\Lambda)$, where the function $\gamma=S^{-1}g\in\lrd$ is  the canonical dual window of  $g$, is a Gabor frame and we have the reproducing formula 
\begin{equation}\label{RF}
 f=\sum_{\lambda\in\Lambda}\langle f,\pi(\lambda)g\rangle\pi(\lambda)\gamma=\sum_{\lambda\in\Lambda}\langle f,\pi(\lambda)\gamma\rangle \pi(\lambda)g,\quad f\in\lrd
\end{equation}
with unconditional convergence in $L^2(\rd)$. In particular, if $\gamma=g$ and $\|g\|_2=1$ then $A=B=1$, the frame operator is the identity $S=I$ and the Gabor frame is called \emph{Parseval} Gabor frame. Observe that formula \eqref{Parseval} holds true.
 More generally, any window function $\gamma\in\lrd$, such that \eqref{RF} is satisfied, is called alternative dual window for $g$. 
 In general, given two functions $g,\gamma\in\lrd$, it is customary  to extend the notion of Gabor frame operator $S_{g,\gamma}$, related to $g,\gamma$, as follows
 $$S_{g,\gamma} f= \sum_{\lambda\in\Lambda}\langle f,\pi(\lambda)g\rangle\pi(\lambda)\gamma,\quad f\in\lrd,
 $$ 
 whenever the previous operator is well-defined. With this notation the reproducing formula \eqref{RF} can be rephrased as  $S_{g,\gamma}=I$ on $\lrd$, with $I$ being the identity operator.

Modulation spaces  provide a natural setting for time-frequency analysis, thanks to discrete equivalent norms produced by means of Gabor frames.  The key result is the following (see \cite[Chapter 12]{Grochenig_2001_Foundations} for $1\leq p,q\leq\infty$, and \cite[Theorem 3.7]{Galperin2004} for $0<p,q<1$). 
\begin{theorem}\label{framesmod}
Assume $m\in\mathcal{M}_v(\rdd)$, $\Lambda=\a \zd\times \b \zd$,  $g,\gamma\in \cS(\rd)$ such that $S_{g,\gamma}=I$ on $\lrd$. Then 
\begin{equation}\label{RFmpq}
f=\sum_{\lambda\in\Lambda}\langle f,\pi(\lambda)g\rangle\pi(\lambda)\gamma=\sum_{\lambda\in\Lambda}\langle f,\pi(\lambda)\gamma\rangle \pi(\lambda)g,\quad f\in M^{p,q}_m(\rd)
\end{equation}
with unconditional convergence in $M^{p,q}_m(\rd)$ if $0<p,q<\infty$ and with weak-* convergence in $M^\infty_{1/v}(\rd)$ otherwise. Furthermore, there are constants $0<A\leq B$ such that, for all $f\in M^{p,q}_m(\rd)$, 
 \begin{equation}\label{gaborframe2}
 A\|f\|_{M^{p,q}_m}\leq\left(\sum_{n\in\zd}\left(\sum_{k\in\zd}|\langle f,\pi(\a k,\beta n)g\rangle|^p m(\a k, \beta n)^p\right)^{\frac qp}\right)^{\frac1q}\leq B\|f\|_{M^{p,q}_m},
 \end{equation}
 independently of $p,q$, and $m$. Similar inequalities hold with $g$ replaced by $\gamma$.
\end{theorem}
In other words, $$\|f\|_{M^{p,q}_m(\rd)}\asymp \|(\la f ,\pi(\lambda)g\ra)_\lambda \|_{\ell^{p,q}_{m}(\Lambda)}= \|(V_g f(\lambda))_\lambda\|_{\ell^{p,q}_{m}(\Lambda)}.$$
\section{Main Results}
We first study convolution relations for modulations spaces. Let us recall that, for the Banach cases, convolution relations were studied in \cite{EleCharly2003} and \cite{toft1,Toftweight2004}. Our approach is general, the techniques use Gabor frames via  the equivalence \eqref{gaborframe2}, plus  H\"{o}lder's and Young's  inequalities for sequences. 
\begin{proposition}\label{mconvmp}
	Let $\nu (\omega )>0$ be  an arbitrary  weight function on $\Ren$,  $0<
	p,q,r,t,u,\gamma\leq\infty$, with
	\begin{equation}\label{Holderindices}
	\frac 1u+\frac 1t=\frac 1\gamma,
	\end{equation}
	and 
	\begin{equation}\label{Youngindicesrbig}
	\frac1p+\frac1q=1+\frac1r,\quad \,\, \text{ for } \, 1\leq r\leq \infty
	\end{equation}
	whereas
	\begin{equation}\label{Youngindicesrbig}
	p=q=r,\quad \,\, \text{ for } \, 0<r<1.
	\end{equation}
	For $m\in\mathcal{M}_v(\rdd)$,   $m_1(x)
	= m(x,0) $ and $m_2(\omega ) = m(0,\omega )$ are the restrictions
	to $\Ren\times\{0\}$ and  $\{0\}\times\Ren$, and likewise for $v$. Then 
	\begin{equation}\label{mconvm}
	M^{p,u}_{m_1\otimes \nu}(\Ren)\ast  M^{q,t}_{v_1\otimes
		v_2\nu^{-1}}(\Ren)\hookrightarrow M^{r,\gamma}_m(\Ren)
	\end{equation}
	with  norm inequality  $$\| f\ast h \|_{M^{r,\gamma}_m}\lesssim
	\|f\|_{M^{p,u}_{m_1\otimes \nu}}\|h\|_{ M^{q,t}_{v_1\otimes
			v_2\nu^{-1}}}.$$
\end{proposition}

\begin{proof}
	We use the key idea in \cite[Proposition 2.4]{EleCharly2003} to measure the \modsp\ norm with respect to the Gaussian windows
	$g_0(x) =e^{-\pi{x^2}}$ and
	$g(x)=2^{-d/2}e^{-\pi x^2/2} = (g_0\ast g_0)(x)\in\sch(\Ren)$.

	A straightforward computation shows $V_g f\phas =e^{-2\pi i x\o}(f\ast M_\o g^*)(x)$ (recall that $g^*(x)=\overline{g(-x)}$).  Thus, using the same argument as in the proof \cite[Proposition 2.4]{EleCharly2003}, we recall the easily verified  identity $M_\o(g_0^*\ast
	g_0^*)=M_\o g_0^*\ast  M_\o g_0^*$ and write the STFT of $f \ast
	h$ as follows:
	$$
	V_g(f\ast h)\phas =e^{-2\pi i x\o}\big( (f\ast h)\ast M_\o g^*\big)(x)
	=e^{-2\pi i x\o}\big( (f\ast M_\o g_0^*) \ast (h\ast M_\omega g_0^*)
	\big)(x) \, .$$
	In the sequel,  we first use the norm equivalence \eqref{gaborframe2}, 
 written in terms of the STFT as
	$
	\|f\|_{\Mmpq } \asymp \|(V_{g} f (\lambda))_{\lambda\in\Lambda}\|_{\ell^{p,q}_{{m}}(\Lambda)},
	$
	where $\Lambda=\a\zd\times\b\zd$. Then we majorize $m$ by
	$$m(\a k, \b n) \lesssim \linebreak[3] m(\a k,0)v(0,\b n ) = m_1(\a k) v_2(\b n),$$
	and finally   use  Young's convolution 
	inequality for sequences in the $k$-variable  and
	H\"older's one in the $n $-variable. The indices $p,q,r,\gamma,t,u$ fulfil the equalities in the assumptions.  We show in details the case when $r,\gamma,t,u<\infty$:
	\begin{align*}
		\|f\ast h\|_{M^{r,\gamma}_m}&\asymp  \|((V_{g}(f\ast h))(\a k,\b n)m(\a k, \b n))_{k,n}\|_{\ell^{r,\gamma}(\zdd)}\\
		&\lesssim \left( \sum_{n\in\zd} \left( \sum_ {k\in\zd} |(f\ast M_{\b n} g_0^*) \ast (h\ast
		M_{\b n} g_0^*) (\a k)|^r m_1(\a k) ^r \,\right) ^{\gamma/r} \, v_2(\b n) ^\gamma
	 \right)^{1/\gamma} \\
		&= \left( \sum_{n\in\zd} \| (f\ast M_{\b n} g_0^*) \ast (h\ast
		M_{\b n} g_0^*) \|^\gamma_{\ell^{r}_{ m_1}(\a\zd)} v_2(\b n)^\gamma  
		\right)^{1/\gamma}\\
		&\lesssim  \left(\sum_{n\in\zd}\|f\ast M_{\b n} g_0^*\|_{\ell^{p}_{m_1}(\a\zd)}^\gamma\,
		\|h\ast M_{\b n} g_0^*\|_{\ell^{q}_{v_1}(\a\zd)}^\gamma\, v_2(\b n)^\gamma\,\right)^{1/\gamma}\\
		&\lesssim  \left(\sum_{n\in\zd} \!\!\|f\ast M_{\b n} g_0^*\|_{\ell^{p}_{m_1}(\a\zd)}^{u} \nu
		(\b n )^{u}   \right) ^{\frac{1}{u}} \!
		\left( \sum_{n\in\zd} \!\!\|h\ast M_{\b n} g_0^*\|_{\ell^{q}_{v_1}(\a\zd)}^{t}\frac{v
			_2(\b n)^{t}}{\nu(\b n)^{t}}\! \right) ^{\frac{1}{t}} \\
		&= \|((V_{g_0}f)(\lambda))_{\lambda}\|_{\ell^{p,u}_{m_1\otimes\nu}(\Lambda)} \,\,		\|((V_{g_0}h)(\lambda))_{\lambda}\|_{\ell^{q,t}_{v_1\otimes v_2\nu \inv}(\Lambda)}\\
		&\asymp  \|f\|_{M^{p,u}_{m_1\otimes \nu }} \,\,
		\|h\|_{M^{q,t}_{v_1\otimes v_2 \nu \inv }}.
	\end{align*}
	The cases when one among the indexes $r,\gamma,t,u$ is equal to $\infty$ are done similarly. This concludes the proof.
\end{proof}

\subsection{Weyl Operators}
Every continuous operator from $\cS (\rd )$ to $\cS ' (\rd )$ can be
represented as a pseudodifferential operator in the  Weyl form $L_\sigma$, with Weyl symbol $\sigma\in\cS'(\rdd)$. The operator  is formally defined by
\begin{equation}\label{Weyl}
L_\sigma f (x)=\intrdd \sigma\left(\frac{x+y}2,\omega\right)e^{2\pi i (x-y)\,\omega}f(y)\,dy d\omega.
\end{equation}
The crucial relation between the action of the Weyl operator $L_\sigma$ on time-frequency shifts and the short-time Fourier transform of its symbol is contained in \cite[Lemma 3.1]{Grochenig_2006_Time}.
\begin{lemma}\label{lemma41} Consider  $g\in \cS(\rd)$, $\Phi=W(g,g)$. Then, for $\sigma\in \cS'(\rdd)$,
	\begin{equation}\label{311}
	|\la L_\sigma \pi(z)g,\pi(w) g\ra|=\left|V_\Phi \sigma\left(\frac{z+w}2,j(w-z)\right)\right|, \quad z,w\in\rdd,
	\end{equation}
	where $j(z_1,z_2)=(z_2,-z_1)$.
\end{lemma}

We first recall Schatten class results for the Weyl calculus in terms of modulation spaces, initially proved for $1\leq p\leq \infty$ in \cite[Theorem 4.5]{BCG02}, for $0<p<1$ we refer to  \cite[Theorem 3.4]{ToftquasiBanach2017}. 
\begin{theorem}\label{Weylcomp}
	 If the Weyl symbol $\sigma\in M^{p,1}(\rdd)$  for some $0< p<\infty$, then the operator $L_\sigma$ belongs to the Schatten class $S_p$  with
	 $$ \|L_\sigma \|_{S_p}\leq \|\sigma\|_{M^{p,1}}.$$
	 In particular, $L_\sigma$ is a compact operator on $\lrd$.
\end{theorem}

\begin{theorem}\label{Charpseudo}
(i)	Assume   $p,q,\gamma\in (0,\infty]$  such that \begin{equation}\label{indicitutti}
	\frac1{p} + \frac{1}{q}=\frac1{\gamma}.
	\end{equation}
If $\sigma \in  M^{p,\min\{1,\gamma\}}(\rdd)$, then the
pseudodifferential operator $L_\sigma$, from $\cS(\rd)$ to $\cS'(\rd)$,
extends uniquely to a bounded operator from ${M}^{q}(\R^d)$ to
${M}^{\gamma}(\R^d)$.\\
(ii) If $s,r\geq 0$, $t\geq r +s$, and the symbol $\sigma \in M^{\infty,1}_{v_s\otimes v_t}(\R^{2d})$, then the
pseudodifferential operator $L_\sigma$, from $\cS(\rd)$ to $\cS'(\rd)$,  extends uniquely to a bounded operator from 
$M^2_{v_r}(\rd)$ into $M^2_{v_{r+s}}(\rd)$.
\end{theorem}
\begin{proof} 
$(i)$ Assume  $\gamma\geq 1$, then, by 	\eqref{indicitutti}, $p\geq \gamma\geq 1$ and $q\geq \gamma\geq 1$. The claim was  proved by Toft in \cite[Theorem 4.3]{toft1}.
The case $\gamma<1$, $p,q\in (0,\infty]$ was again proved by Toft in \cite[Theorem 3.1]{ToftquasiBanach2017}.

$(ii)$ Let $g \in \cS(\rd)$ with $\| g \|_{L^2}=1$.
From the inversion formula \eqref{invformula},
\[
V_g (L_\sigma f)(w)=\int_{\R^{2d}}\langle L_\sigma \pi(z) g,
\pi(w)g\rangle \, V_g f(z) \, dz.
\]
The desired result thus follows if we can prove that the map $M(\sigma)$
defined by
\[
M(\sigma) G(w)=\int_{\R^{2d}}\langle  L_\sigma \pi(z) g,
\pi(w) g \rangle \, G(z) \, dz
\]
is continuous from $L^{2}_{v_r}(\rdd)$ into $L^{2}_{v_{r+s}}(\rdd)$. Using \eqref{311}, we see that it is sufficient to prove that the integral operator with integral kernel 
\[
 \left|V_\Phi \sigma\left(\frac{z+w}2,j(w-z)\right)\right|\langle z\rangle^{-r} \langle w\rangle^{r+s}
\]
is bounded on $L^2(\rdd)$. This follows from Schur's test (see, e.g., \cite[Lemma 6.2.1 (b)]{Grochenig_2001_Foundations}). Indeed, by assumption $\sigma\in M^{\infty,1}_{v_s\otimes v_t}$, so that 
\[
\sup_{w\in\rdd} \intrdd  \left|V_\Phi \sigma\left(\frac{z+w}2,j(w-z)\right)\right| \la z+w \ra^{s}  \la w- z\ra^{t} dz<\infty
 \]
 and similarly 
 \[
 \sup_{z\in\rdd} \intrdd  \left|V_\Phi \sigma\left(\frac{z+w}2,j(w-z)\right)\right| \la z+w \ra^{s}  \la w- z\ra^{t} dw<\infty.
 \]
Hence it is sufficient to prove that for some positive constant $C>0$ we have
\begin{equation}\label{estF1}
\la z+w\ra^{-s}\la w-z\ra ^{-t} \la z \ra^{-r} \la w\ra^{r+s}\leq C,\quad \forall w,z\in\rdd.
\end{equation}
Let us prove the estimate \eqref{estF1}. Setting $x=z+w$, $y=w-z$, the inequality \eqref{estF1} can be rephrased as
\begin{equation}\label{estF2}
\la x\ra^{-s}\la y\ra ^{-t} \la x-y\ra^{-r} \la x+y\ra^{r+s}\leq C,\quad \forall x,y\in\rdd.
\end{equation}
For $|x|< 2|y|$, observe that $|x+y|< 3|y|$ and since $t\geq r+ s$ we get the estimate \eqref{estF2}. For $|x|\geq 2|y|$, we use $\la x+y\ra \asymp \la x-y\ra \asymp \la x \ra$ and  \eqref{estF2} immediately follows.
\end{proof}

 We remark that \begin{eqnarray*}
 	M^2_{\vs}(\Ren)= \Qs(\Ren),
 \end{eqnarray*}
  the Shubin-Sobolev spaces, cf. \cite{BCG02,Shubin91}. In particular, for $s=0$, $M^2(\rd)=L^2(\rd)$.  Thus, we recover the known result, c.f. \cite[Theorem 4.3]{Toftweight2004}:
 \begin{corollary}
 If $s,r\geq 0$, $t\geq r +s$, and the symbol $\sigma \in M^{\infty,1}_{v_s\otimes v_t}(\R^{2d})$, then the
 pseudodifferential operator $L_\sigma$, from $\cS(\rd)$ to $\cS'(\rd)$,  extends uniquely to a bounded operator from 
 $Q_r(\Ren)$ into ${Q}_{{r+s}}(\rd)$.
 \end{corollary}
 
An application of the previous theorem concerns the study of eigenfunctions' properties for Weyl operators. Observe that a Weyl operator having symbol $\sigma \in M^{p,\gamma}(\rdd)$ with $p<\infty$ is a compact operator on $\lrd$ by Theorem \ref{Weylcomp}.
\begin{proposition}\label{eigenfunctWeyl}
Consider a Weyl symbol $\sigma\in M^{p,\gamma}$ for some $0< p<\infty$ and every $\gamma>0$.  If $\lambda\in \sigma_P(L_\sigma)\setminus\{0\}$, then any eigenfunction $f\in L^2(\rd)$ with eigenvalue $\lambda$  is in $\cap_{\gamma>0}M^\gamma(\rd)$.
\end{proposition}
\begin{proof}
	By Theorem \ref{Charpseudo}, if the symbol $\sigma$ is in $M^{p,\gamma}(\rdd)$, for every $\gamma>0$, then the Weyl operator acts continuously from $ M^2(\rd)$ into $M^{\gamma_1}(\rd)$, with  $1/p+1/2=1/{\gamma_1}$ and, since $p<\infty$, we have $\gamma_1<2$.
	Thus, for $f\in M^2(\rd)$ eigenfunction with eigenvalue $\lambda\not=0$, we have $f=\frac1\lambda L_\sigma f\in M^{\gamma_1}(\rd)$. Starting with $f\in M^{\gamma_1}(\rd)$, we repeat the same argument, obtaining that the eigenfunction $f$ is in the smaller modulation space $M^{\gamma_2}(\rd)$, with 
	$$\frac{1}{p}+\frac{1}{\gamma_1}= \frac{1}{\gamma_2}
	$$
(observe $\gamma_2<\gamma_1$ since $p<\infty$). Continuing this way we construct a decreasing sequence of indices $\gamma_n>0$ and such that $\lim_{n\to\infty} \gamma_n=0$.
This proves the claim.
\end{proof}
Recall that, for $s,t>0$, by inclusion relations for modulation spaces we have $ M^{\infty,1}_{v_s\otimes v_t}(\rdd)\subset M^{p,1}(\rdd)$, for any $p>2d/s$, \cite[Theorem 1.5 and Lemma 4.10]{Guo2019}. Hence the operators below are compact operators by Theorem \ref{Weylcomp}.
\begin{proposition}\label{eigenfunctWeyl2}
	Consider   a Weyl symbol $\sigma\in M^{\infty,1}_{v_s\otimes v_t}(\rdd)$ for some $s>0$ and every $t>0$. If $\lambda\in \sigma_P(L_\sigma)\setminus\{0\}$, then any eigenfunction $f\in L^2(\rd)$ with eigenvalue $\lambda$ is in $\cS(\rd)$.
\end{proposition}
\begin{proof}
	By Theorem \ref{Charpseudo}, if the symbol $\sigma$ is in $ M^{\infty,1}_{v_s\otimes v_t}(\rdd)$, for some $s>0$ and every $t>0$, then the Weyl operator acts continuously from $ L^2(\rd)$ into $M^{2}_{v_{s}}(\rd)=Q_{s}(\rd)$.
	Starting now with the eigenfunction $f$  in $Q_{s}(\rd)$ and repeating the same argument with $t\geq s$ we obtain that the eigenfunction is in $Q_{2s}(\rd)$. Proceeding this way we infer that $f\in \cap_{n\in\N_+} Q_{ns}(\rd)$.
	The inclusion relations for Shubin-Sobolev spaces and the property (see e.g., \cite{EleCharly2003,feichtinger-grochenig1997,Toftweight2004}) $$\cap_{r>0}Q_r(\rd)=\cS(\rd),$$
	 prove the claim.
\end{proof}

	\subsection{Localization Operators}
	The study of eigenfunctions for a localization operator $\aaf $ uses its  representation  as a 	Weyl one:
	  $$\gaw=L_{a\ast W(\f_2,\f_1)}$$ 
	(cf. \cite{EleCharly2003} and references therein) where $W(\f_2,\f_1)$ is the the cross-Wigner distribution defined in \eqref{CWD}.
	Therefore, 	the Weyl symbol of $\aaf$ is given by
	\begin{equation}\label{eq2}
	\sigma = a\ast W(\f_2,\f_1).
	\end{equation}
	We will deduce properties for $\gaw$ using its  Weyl form $L_{a\ast W(\f_2,\f_1)}$, as detailed below. Precisely, we shall focus on properties of eigenfunctions of compact localization operators. 
	\begin{theorem}\label{main}
		 Consider a symbol $a\in M^{p,\infty}(\rdd)$, for some  $ 0< p< \infty$, and non-zero windows $\f_1,\f_2\in\cS(\rd)$.
		 Assume that $\sigma_P(\aaf)\setminus\{0\}\not=\emptyset$. If $\lambda\in \sigma_P(\aaf)\setminus\{0\}$,   any eigenfunction $f\in\lrd$ with eigenvalue  $\lambda$  satisfies $f\in \bigcap_{\gamma>0} M^\gamma(\rd)$.
	\end{theorem}
	\begin{proof}
		Since the windows $\f_1,\f_2\in\cS(\rd)$, the cross-Wigner distribution is in $\cS(\rdd)\subset M^{q,\gamma}(\rdd)$, for every $0<q,\gamma\leq \infty$.
		We next apply the convolution relations for modulation spaces \eqref{mconvm}. Namely, if $p\geq 1$, choose  $q=1$, so that $r=p$,  and
		$$\sigma = a\ast W(\f_2,\f_1)\in  M^{p,\infty}(\rdd)\ast M^{1,\gamma}(\rdd)\hookrightarrow M^{p,\gamma}(\rdd),\quad \forall \gamma>0.$$
		If $0<p<1$, choose $p=q=r$ so that
		$$\sigma = a\ast W(\f_2,\f_1)\in  M^{p,\infty}(\rdd)\ast M^{p,\gamma}(\rdd)\hookrightarrow M^{p,\gamma}(\rdd),\quad \forall \gamma>0.$$
		 In both cases we obtain that $\sigma \in M^{p,\gamma}(\rdd)$, for every $\gamma>0$. Hence the claim immediately follows by Proposition \ref{eigenfunctWeyl}. 
	\end{proof}

As a consequence, the eigenfunctions  are extremely concentrated on the time-frequency space, having very few Gabor coefficients large whereas all the others are negligible.

Consider a \emph{Parseval} Gabor frame $\G(g,\Lambda)$ for $\lrd$, with  $\Lambda=\a \zd\times \beta\zd$ and $g \in \cS(\rd)$. For $N\in \bN_+$, we defined in \eqref{Sigma} the space $\Sigma_N$ 
to be the set of all linear combinations of Gabor atoms consisting of at most $N$ terms. 

Given a function $f\in\lrd$, the $N$-term approximation error in $\lrd$ is recalled in \eqref{sigma}.
Namely, $\sigma_N(f)$ is the error produced when $f$ is approximated optimally by a linear combination of $N$ Gabor atoms.  

Assume $f\in M^p(\rd)$ for some $0<p<2$ (thus, in particular,  $f\in\lrd$).  The series of Gabor coefficients in \eqref{Parseval} are absolutely convergent, hence also unconditionally convergent. Thus we can rearrange the Gabor coefficients $|\la f ,\pi(\lambda)g\ra|$ in a decreasing order. Precisely, set $c_{k,n}=\la f ,\pi(\a k, \beta n)g\ra$, $k,n\in \zd$,  and let $\iota: \bN_+ \to \zd\times \zd$ be any bijection  satisfying 
$$|c_{\iota(1)}|\geq |c_{\iota(2)}|\geq \cdots \geq |c_{\iota(m)}|\geq |c_{\iota(m+1)}|\geq \cdots
$$
The sequence $(\tilde{c}_m)_{m\in\bN_+} =(|c_{\iota(m)}|)_{m\in\bN_+}$ is called the non-increasing rearrangement of $(c_{k,n})_{k,n}$ above. With this notations, denoting by $p_{opt}$ the best approximation of $f$ in $\Sigma_N$, 
 the $N$-term approximation error becomes
\begin{align}\sigma_N(f)&=\inf_{p\in\Sigma_N} \|f-p\|_2=\|f-p_{opt}\|_2\leq \left\|f- \sum_{m=1}^{N}c_{\iota(m)}\pi(\iota(m))g\right\|_2\label{F12}\\
&\leq  C \left(\sum_{m=N+1}^\infty |c_{\iota(m)}|^2\right)^\frac12.\notag
\end{align}

By abuse of notation, given $a=(a_m)_m$, $a_m\geq 0$  for every $m$, a non-increasing sequence ($a_1\geq a_2\geq \cdots \geq a_m\geq a_{m+1}\geq \cdots$) we write
$$\sigma_N(a)=\left(\sum_{m=N+1}^\infty a_m^2\right)^\frac{1}{2}.
$$
The key tool is now the following lemma, see \cite {stechkin55} and \cite{devore-temlyakov96} (we also refer to \cite[Lemma 12.4.1]{Grochenig_2001_Foundations}):
\begin{lemma}\label{deVore}
	Let $a=(a_m)_m$, $a_m\geq 0$ for every $m$, be a non-increasing sequence and consider $0<p<2$. Set 
	\begin{equation}\label{alfadef}
	\gamma= \frac1p-\frac 12>0.
	\end{equation}
	Then there exists a constant $C=C(p)>0$, such that 
	\begin{equation}\label{est1}
	\frac{1}{C}\|a\|_{\ell^p}\leq \left(\sum_{N=1}^\infty(N^\gamma \sigma_{N-1}(a))^p\frac{1}{N}\right)^\frac1p\leq C\|a\|_{\ell^p}.
	\end{equation}
\end{lemma}
\begin{proposition}\label{approx}
	Assume $f\in M^p(\rd)$ for some $0<p<2$. Then, there exists $C=C(p)>0$ such that the $N$-term approximation error satisfies
	\begin{equation}\label{approx2}
	\sigma_N(f)\leq C\|f\|_{M^p(\rd)}N^{-\gamma},
	\end{equation}
	where $\gamma>0$ is defined in \eqref{alfadef}.
\end{proposition}
\begin{proof}
	If $\G(g,\Lambda)$ is a Parseval Gabor frame with $g \in \cS(\rd)$, and $\Lambda=\a \zd\times \beta \zd$, $\alpha,\beta>0$, then the sequence of Gabor  coefficients of $f$, given by $(\la f ,\pi(\a k,\beta n)g\ra)_{k,n\in\zd}$, are in $\ell^{p}(\Lambda)$ by Theorem \ref{framesmod}, with $$\|f\|_{M^p} \asymp \|(\la f ,\pi(\a k,\beta n)g\ra)_{k,n}\|_{\ell^p(\Lambda)}$$ and  the sequence $(|\la f ,\pi(\a k,\beta n)g\ra|)_{k,n}$ can be rearranged in a non-increasing one $(a_m)$, as  explained above. Applying Lemma \ref{deVore} to such a sequence, from the right-hand side inequality in \eqref{est1} and using the majorization in \eqref{F12} we infer \eqref{approx2}.
	\end{proof}
	
\begin{corollary}\label{bo1}
	Consider a Parseval Gabor frame $\G(g,\Lambda)$, with $g \in \cS(\rd)$, and $\Lambda=\a \zd\times \beta \zd$, $\alpha,\beta>0$. Under the assumptions of Theorem \ref{main}, any $f$ eigenfunction of $\aaf$ (with eigenvalue $\lambda\not=0$) is highly compressed onto a few Gabor atoms $\pi(\lambda)g$, in the sense that its $N$-term approximation error satisfies the following property: for every $r>0$ there exists $C=C(r,f)>0$ with  
	\begin{equation}\label{bo2}
\sigma_N(f)\leq  C N^{-r}.
	\end{equation}
	\end{corollary}
	\begin{proof}
		By Theorem \ref{main}, the eigenfunction fulfils $f\in M^{p}(\rd)$, for every $p>0$. Hence the assumptions of Proposition \ref{approx} are satisfied for every $0<p<2$. This immediately yields the claim.
	\end{proof}

We next consider the case of compact localization operators with symbols $a\in M^{\infty}_{v_s\otimes 1}(\rdd)$, $s>0$. In this case $L^2$ eigenfunctions reveal to be Schwartz functions, as shown below.
\begin{theorem}\label{main2}
	Consider a symbol $a\in M^{\infty}_{v_s\otimes 1}(\rdd)$, for some  $ s>0$, and non-zero windows $\f_1,\f_2\in\cS(\rd)$.
	Assume that $\sigma_P(\aaf)\setminus\{0\}\not=\emptyset$. If $\lambda\in \sigma_P(\aaf)\setminus\{0\}$,   any eigenfunction $f\in\lrd$ with eigenvalue  $\lambda$  belongs to the Schwartz class 
	$\cS(\rd)$.
\end{theorem}
\begin{proof}
The assumption $\f_1,\f_2\in\cS(\rd)$ implies $W(\f_2,\f_1)\in \cS(\rdd)\subset M^{1,1}_{v_r\otimes v_t}(\rdd)$, for every $r,t>0$.
	We next apply the convolution relations for modulation spaces \eqref{mconvm}, obtaining that $\aaf=L_\sigma$ with $\sigma \in M^{\infty,1}_{v_s\otimes v_t}(\rdd)$, for some $s>0$ and every $t>0$. Hence the claim immediately follows by Proposition \ref{eigenfunctWeyl2}. 
\end{proof}
\begin{remark}\label{schatten-comp}
	The nice properties of eigenfunctions for localization operators studied so far seem to depend on the fact that such operators are not only compact but belong to the Schatten class $S_p$, $0<p<\infty$ (cf. \cite[Theorem 1]{ECSchattenloc2005}). Fern\'{a}ndez and  Galbis in \cite[Theorem 3.15]{FG2006} characterize compact localization operators. Namely, fix $g_0\in\cS(\rd)$ and $a\in M^\infty(\rdd)$, then   the following conditions are equivalent:\\
	(i) The localization operator $\aaf$ is compact on $\lrd$ for every $\f_1,\f_2$ in $\cS(\rd)$;\\
	(ii) For every $R>0$,
	\begin{equation}\label{compsymb}
	\lim_{|x|\to \infty }\sup_{|\xi|\leq R} V_{g_0} a\phas =0 .
	\end{equation}
	It seems that for symbols satisfying condition \eqref{compsymb} the techniques developed above do not work anymore. It would be very interesting to know whether for compact operators that are not in the Schatten class $S_p$, $0<p<\infty$, the $L^2$ eigenfunctions do gain any additional smoothness and regularity.
	This topic will be investigated in a subsequent paper.
\end{remark}

%
\section{Symbols in $L^p(\rdd)$ spaces}
We now consider localization operators with symbols in weighted Lebesgue spaces. Let us recall 
that any localization operator $\gaw$ with windows in $\cS(\rd)$ and symbol $a\in L^q(\rdd)$, with $1\leq q<\infty$, is a compact operator, cf. \cite[Proposition 13.3]{WongLocalization}. The case of weighted Lebesgue spaces and, more generally, Potential Sobolev spaces was treated in \cite{BCG02}: let us stress that any localization operator $\aaf$ with Schwartz windows and symbol $a$ in  $L^q_m(\rdd)$, with $1\leq q<\infty$ is a compact operator on $\lrd$.

\begin{theorem}\label{mainLp}
	Let $m\in\mathcal{M}_v$, $m(z)\geq 1$ for every $z\in \rdd$, $a\in L^q_m(\rdd)$, $1\leq q< \infty$, and non-zero windows $\f_1,\f_2\in\cS(\rd)$. Assume that $\sigma_P(\aaf)\setminus\{0\}\not=\emptyset$. If $\lambda\in \sigma_P(\aaf)\setminus\{0\}$,   any eigenfunction $f\in\lrd$ with eigenvalue  $\lambda$  satisfies 
	$f\in \bigcap_{p>0} M^p_m(\rd)$.
\end{theorem}
\begin{proof}
	By assumption and using \eqref{pWiener}, we start with a symbol $a$  in  $L^q_m(\rdd)=W(L^q,L^q_m)(\rdd)$. Consider the eigenvector $f\in L^2(\rd)$ and the window $\f_1\in\cS(\rd)$. Then   by Theorem \ref{G33} the STFT $V_{\f_1} f$ is in the Wiener amalgam space $W(L^\infty, L^2)(\rdd)$. Proposition \ref{MR} yields that $aV_{\f_1} f\in W(L^q,L^{p_1}_m)(\rdd)$, with
	$$\frac 1q+\frac12=\frac{1}{p_1},$$
	so that the index $p_1$ satisfies $p_1<\min\{q,2\}$.  Consider now a non-zero window $g\in\cS(\rd)$. Using \eqref{CR},
	\begin{align*}
	V_g (\aaf f)(w)&=\langle \aaf f, \pi(w)g\ra =\intrdd \left(aV_{\f_1} f\right)(z)\la \pi(z)\f_2, \pi(w)g\ra \,dz\\
	&= \intrdd \left(aV_{\f_1} f\right)(z)\la \f_2, \pi(-z)\pi(w)g\ra \,dz\\
	&=\intrdd \left(aV_{\f_1} f\right)(z) e^{-2\pi i z_1 w_2}\la \f_2, \pi(w-z)g\ra \,dz
	\end{align*}
	so that,
	\begin{equation}\label{convloc}
	|V_g (\aaf f)(w)|\leq \intrdd |(a V_{\f_1}f)(z)|\,|V_g \f_2 (w-z)| dz =|a V_{\f_1}f|\ast |V_g \f_2|(w) .
	\end{equation}
		We estimate $$|V_g (\aaf f)(w)|\leq |a V_{\f_1}f|\ast |V_g \f_2|(w)\in W(L^q,L^{p_1}_m)(\rdd)\ast W(L^\infty,L^1_v)(\rdd).$$ Observing that $W(L^\infty,L^1_v)(\rdd)\hookrightarrow W(L^{q'},L^1_v)(\rdd)$  and applying the convolution relations \eqref{convWiener} we infer
	$|V_g (\aaf f)| \in W(L^\infty,L^{p_1}_m) \hookrightarrow L^{p_1}_m$. This proves that 
	$\aaf f\in M^{p_1}_m(\rd)$.
	
	Recalling the assumption $\aaf f=\lambda f$, $\lambda \not=0$, we infer $f\in M^{p_1}_m(\rd)$.
	
	We now repeat the previous argument starting with $f\in M^{p_1}_m(\rd)$.  By Theorem \ref{G33} the STFT $V_{\f_1} f\in W(L^\infty, L^{p_1}_m)(\rdd)$  and $aV_{\f_1} f\in W(L^q, L^{p_2}_{m^2})\hookrightarrow W(L^q, L^{p_2}_m)$, (since $m^2\geq m$), with
	$$ \frac 1q+ \frac{1}{p_1}=\frac{1}{p_2}, 
	$$
	so that $ p_2<p_1.$ Arguing as above we infer  $|V_g (\aaf f)(w)|\in W(L^\infty,L^{p_2}_m) \hookrightarrow L^{p_2}_m$. Thus, the eigenfunction $f$ belongs to the smaller space $M^{p_2}_m$.
	
	Continuing this way we construct a strictly decreasing sequence of indices $p_n>0$ and such that $$\lim_{n\to \infty} p_n =0.$$
	By induction and using the same argument as above one immediately obtains that if $f\in 
	M^{p_n}_m(\rd)$ then $f\in M^{p_{n+1}}_m(\rd)$. This concludes the proof.
\end{proof}

\section*{Acknowledgments}  The  last two authors were partially supported by  the Gruppo Nazionale per l'Analisi Matematica, la Probabilit\`a e le loro Applicazioni (GNAMPA) of the Istituto Nazionale di Alta Matematica (INdAM).


\begin{thebibliography}{10}
	
	\bibitem{Abreu2012}
	L.~D. Abreu and M.~D\"orfler.
	\newblock An inverse problem for localization operators.
	\newblock {\em Inverse Problems}, 28(11):115001, 16, 2012.
	
	
	\bibitem{Abreu2016}
	L.~D. Abreu, K.~Gr\"{o}chenig, and J.~L. Romero.
	\newblock On accumulated spectrograms.
	\newblock {\em Trans. Amer. Math. Soc.}, 368(5):3629--3649, 2016.
	
	\bibitem{Abreu2017}
	L.~D. Abreu, J.~a.~M. Pereira, and J.~L. Romero.
	\newblock Sharp rates of convergence for accumulated spectrograms.
	\newblock {\em Inverse Problems}, 33(11):115008, 12, 2017.
	
	\bibitem{BG2015}
	D.~Bayer and K.~Gr\"{o}chenig.
	\newblock Time-frequency localization operators and a {B}erezin transform.
	\newblock {\em Integral Equations Operator Theory}, 82(1):95--117, 2015.
	
	\bibitem{Berezin71}
	F.~A. Berezin.
	\newblock Wick and anti-{W}ick symbols of operators.
	\newblock {\em Mat. Sb. (N.S.)}, 86(128):578--610, 1971.
	
	\bibitem{BCG02}
	P.~Boggiatto, E.~Cordero, and K.~Gr\"{o}chenig.
	\newblock Generalized anti-{W}ick operators with symbols in distributional
	{S}obolev spaces.
	\newblock {\em Integral Equations Operator Theory}, 48(4):427--442, 2004.
	
	\bibitem{EleCharly2003}
	E.~Cordero and K.~Gr\"ochenig.
	\newblock Time-frequency analysis of localization operators.
	\newblock {\em J. Funct. Anal.}, 205(1):107--131, 2003.
	
	\bibitem{ECSchattenloc2005}
	E.~Cordero and K.~Gr\"ochenig.
	\newblock Necessary conditions for {S}chatten class localization operators.
	\newblock {\em Proc. Amer. Math. Soc.}, 133(12):3573--3579, 2005.
	
	\bibitem{Wignersharp2018}
	E.~Cordero and F.~Nicola.
	\newblock Sharp integral bounds for {W}igner distributions.
	\newblock {\em Int. Math. Res. Not. IMRN}, (6):1779--1807, 2018.
	
	\bibitem{medit}
	E.~Cordero, S.~Pilipovi\'c, L.~Rodino, and N.~Teofanov.
	\newblock Localization operators and exponential weights for modulation spaces.
	\newblock {\em Mediterr. J. Math.}, 2(4):381--394, 2005.
	
	\bibitem{DB1}
	I.~Daubechies.
	\newblock Time-frequency localization operators: a geometric phase space
	approach.
	\newblock {\em IEEE Trans. Inform. Theory}, 34(4):605--612, 1988.
	
	\bibitem{DB2}
	I.~Daubechies and T.~Paul.
	\newblock Time-frequency localization operators---a geometric phase space
	approach. {II}. {T}he use of dilations.
	\newblock {\em Inverse Problems}, 4(3):661--680, 1988.
	
	\bibitem{deGossonsymplectic2011}
	M.~A. de~Gosson.
	\newblock {\em Symplectic methods in harmonic analysis and in mathematical
		physics}, volume~7 of {\em Pseudo-Differential Operators. Theory and
		Applications}.
	\newblock Birkh\"auser/Springer Basel AG, Basel, 2011.
	
	\bibitem{Maurice2015}
	M.~A. de~Gosson.
	\newblock The canonical group of transformations of a {W}eyl-{H}eisenberg
	frame; applications to {G}aussian and {H}ermitian frames.
	\newblock {\em J. Geom. Phys.}, 114:375--383, 2017.
	
	\bibitem{devore-temlyakov96}
	R.~A. DeVore and V.~N. Temlyakov.
	\newblock Some remarks on greedy algorithms.
	\newblock {\em Adv. Comput. Math.}, 5(2-3):173--187, 1996.
	
	\bibitem{Feichtinger_1981_Banach}
	H.~G. Feichtinger.
	\newblock Banach spaces of distributions of {W}iener's type and interpolation.
	\newblock In {\em Functional analysis and approximation ({O}berwolfach, 1980)},
	volume~60 of {\em Internat. Ser. Numer. Math.}, pages 153--165. Birkh\"auser,
	Basel-Boston, Mass., 1981.
	
	\bibitem{feichtinger-wiener-type}
	H.~G. Feichtinger.
	\newblock Banach convolution algebras of {W}iener type.
	\newblock In {\em Functions, series, operators, Vol. I, II (Budapest, 1980)},
	pages 509--524. North-Holland, Amsterdam, 1983.
	
	\bibitem{Feichtinger_1990_Generalized}
	H.~G. Feichtinger.
	\newblock Generalized amalgams, with applications to {F}ourier transform.
	\newblock {\em Canad. J. Math.}, 42(3):395--409, 1990.
	
	\bibitem{feichtinger-modulation}
	H.~G. Feichtinger.
	\newblock Modulation spaces on locally compact abelian groups.
	\newblock In {\em Technical report, University of Vienna, 1983, and also in
		``Wavelets and Their Applications''}, pages 99--140. M. Krishna, R. Radha, S.
	Thangavelu, editors, Allied Publishers, 2003.

	\bibitem{feichtinger-grochenig1997}
	H.~G. Feichtinger and K.~Gr{\"o}chenig.
	\newblock Gabor frames and time-frequency analysis of distributions.
	\newblock {\em J. Funct. Anal.}, 146(2):464--495, 1997.	
	
	\bibitem{FGM2003}
	H.~G. Feichtinger and K. Nowak.
	\newblock A first survey of {G}abor multipliers.
	\newblock In {\em Advances in {G}abor analysis}, Appl. Numer. Harmon. Anal.,  pages 99--128. Birkh\"{a}user Boston, Boston, MA, 2003.

	\bibitem{FG2006}
	C. Fern\'{a}ndez and A. Galbis.
			\newblock  Compactness of time-frequency localization operators on
			{$L^2(\Bbb R^d)$}.
				\newblock {\em J. Funct. Anal.}, 233(2):335--350, 2006.
		\bibitem{FG2007}
		C. Fern\'{a}ndez and A. Galbis.
			\newblock Some remarks on compact {W}eyl operators. 
			\newblock {\em Integral Transforms Spec. Funct.}, 18(7-8):599--607, 2007.
			\bibitem{FGP2017}
		C. Fern\'{a}ndez,  A. Galbis. and E. Primo.
			\newblock Compactness of {F}ourier integral operators on weighted
				modulation spaces.
			\newblock {\em Trans. Amer. Math. Soc.}, 372(1):733--753, 2019.
	\bibitem{Fournier_1985_Amalgams}
	J.~J.~F. Fournier and J.~Stewart.
	\newblock Amalgams of {$L^p$} and {$l^q$}.
	\newblock {\em Bull. Amer. Math. Soc. (N.S.)}, 13(1):1--21, 1985.
	
	\bibitem{Galperin2014}
	Y.~V. Galperin.
	\newblock Young's convolution inequalities for weighted mixed (quasi-) norm
	spaces.
	\newblock {\em J. Inequal. Spec. Funct.}, 5(1):1--12, 2014.
	
	\bibitem{Galperin2004}
	Y.~V. Galperin and S.~Samarah.
	\newblock Time-frequency analysis on modulation spaces {$M^{p,q}_m$}, {$0<p,\
		q\leq\infty$}.
	\newblock {\em Appl. Comput. Harmon. Anal.}, 16(1):1--18, 2004.
	
	\bibitem{Grochenig_2001_Foundations}
	K.~Gr{\"o}chenig.
	\newblock {\em Foundations of time-frequency analysis}.
	\newblock Applied and Numerical Harmonic Analysis. Birkh\"auser Boston, Inc.,
	Boston, MA, 2001.
	
	\bibitem{Grochenig_2006_Time}
	K.~Gr{\"o}chenig.
	\newblock Time-frequency analysis of {S}j\"ostrand's class.
	\newblock {\em Rev. Mat. Iberoam.}, 22(2):703--724, 2006.
	
	\bibitem{GroLyu2009}
	K.~Gr\"{o}chenig and Y.~Lyubarskii.
	\newblock Gabor (super)frames with {H}ermite functions.
	\newblock {\em Math. Ann.}, 345(2):267--286, 2009.
	
	\bibitem{CharlyToft2011}
	K.~Gr\"{o}chenig and J.~Toft.
	\newblock Isomorphism properties of {T}oeplitz operators and
	pseudo-differential operators between modulation spaces.
	\newblock {\em J. Anal. Math.}, 114:255--283, 2011.
	
	\bibitem{CharlyToft2013}
	K.~Gr\"{o}chenig and J.~Toft.
	\newblock The range of localization operators and lifting theorems for
	modulation and {B}argmann-{F}ock spaces.
	\newblock {\em Trans. Amer. Math. Soc.}, 365(8):4475--4496, 2013.
	
	\bibitem{Guo2019}
	W. Guo, J. Chen, D. Fan and G. Zhao.
 	\newblock	Characterizations of Some Properties on Weighted Modulation and Wiener Amalgam Spaces
 		\newblock {\em Michigan Math. J.}, 68:451--482, 2019.
 	
	\bibitem{Heil-amalgam}
	C.~Heil.
	\newblock An introduction to weighted wiener amalgams.
	\newblock In {\em In: Wavelets and their Applications, M. Krishna, R. Radha and
		S. Thangavelu, eds.,}, pages 183--216. Allied Publishers, New Delhi, 2003.
	
	\bibitem{Kobayashi2006}
	M.~Kobayashi.
	\newblock Modulation spaces {$M^{p,q}$} for {$0<p,q\leq\infty$}.
	\newblock {\em J. Funct. Spaces Appl.}, 4(3):329--341, 2006.
	
	\bibitem{Kobayashi2007}
	M.~Kobayashi.
	\newblock Dual of modulation spaces.
	\newblock {\em J. Funct. Spaces Appl.}, 5(1):1--8, 2007.
	
	\bibitem{Luef1}
	F.~Luef and E.~Skrettingland.
	\newblock Mixed-state localization operators: {C}ohen's class and trace
	class operators.
	\newblock {\em J. Fourier Anal. Appl.}, 25(4):2064--2108, 2019.
	
	\bibitem{Luef2}
	F.~Luef and E.~Skrettingland.
	\newblock On accomulated {C}ohen's class distributions and mixed-state
	localization operators.
	\newblock {\em e-prints, arXiv:1808.06419}, 2018.
	
	\bibitem{Rauhut2007Coorbit}
	H.~Rauhut.
	\newblock Coorbit space theory for quasi-{B}anach spaces.
	\newblock {\em Studia Math.}, 180(3):237--253, 2007.
	
	\bibitem{Rauhut2007Winer}
	H.~Rauhut.
	\newblock Wiener amalgam spaces with respect to quasi-{B}anach spaces.
	\newblock {\em Colloq. Math.}, 109(2):345--362, 2007.
	
	\bibitem{Shubin91}
	M.~A. Shubin.
	\newblock {\em Pseudodifferential operators and spectral theory}.
	\newblock Springer-Verlag, Berlin, second edition, 2001.
	
	\bibitem{stechkin55}
	S.~B. Stechkin.
	\newblock On absolute convergence of orthogonal series.
	\newblock {\em Dokl. Akad. Nauk SSSR}, 102:37--40, 1955.
	
	\bibitem{Nenad2015}
	N.~Teofanov.
	\newblock Gelfand-{S}hilov spaces and localization operators.
	\newblock {\em Funct. Anal. Approx. Comput.}, 7(2):135--158, 2015.
	
	\bibitem{Nenad2016}
	N.~Teofanov.
	\newblock Continuity and {S}chatten--von {N}eumann properties for localization
	operators on modulation spaces.
	\newblock {\em Mediterr. J. Math.}, 13(2):745--758, 2016.
	
	\bibitem{Nenad2018}
	N.~Teofanov.
	\newblock Bilinear localization operators on modulation spaces.
	\newblock {\em J. Funct. Spaces}, pages Art. ID 7560870, 10, 2018.
	
	\bibitem{toft1}
	J.~Toft.
	\newblock Continuity properties for modulation spaces, with applications to
	pseudo-differential calculus. {I}.
	\newblock {\em J. Funct. Anal.}, 207(2):399--429, 2004.
	
	\bibitem{Toftweight2004}
	J.~Toft.
	\newblock Continuity properties for modulation spaces, with applications to
	pseudo-differential calculus. {II}.
	\newblock {\em Ann. Global Anal. Geom.}, 26(1):73--106, 2004.
	
		\bibitem{ToftquasiBanach2017}
		J.~Toft.
		\newblock Continuity and compactness for pseudo-differential operators
			with symbols in quasi-{B}anach spaces or {H}\"{o}rmander classes.
		\newblock {\em Anal. Appl. (Singap.)}, 15(3):353--389, 2017.
	
	\bibitem{Wangbook2011}
	B.~Wang, Z.~Huo, C.~Hao, and Z.~Guo.
	\newblock {\em Harmonic analysis method for nonlinear evolution equations.
		{I}}.
	\newblock World Scientific Publishing Co. Pte. Ltd., Hackensack, NJ, 2011.
	
	\bibitem{WongLocalization}
	M.~W. Wong.
	\newblock {\em Wavelet transforms and localization operators}, volume 136 of
	{\em Operator Theory: Advances and Applications}.
	\newblock Birkh\"auser Verlag, Basel, 2002.
	
\end{thebibliography}
\end{document}